\documentclass[11pt]{article}

\usepackage[a4paper,margin=1.1in]{geometry}
\usepackage{setspace}
\usepackage[T1]{fontenc}
\usepackage{lmodern}
\usepackage{microtype}

\usepackage{amsmath,amssymb,amsthm}
\usepackage{mathtools}
\usepackage{mathrsfs}
 \usepackage{pgf,tikz}
 \usetikzlibrary{automata}
 \usetikzlibrary{arrows}
\usepackage[numbers]{natbib}

\usepackage[hidelinks]{hyperref}

\usepackage{fancyhdr}
\theoremstyle{plain}
\newtheorem{theorem}{Theorem}[section]

\newtheorem{proposition}[theorem]{Proposition}
\newtheorem{corollary}[theorem]{Corollary}

\theoremstyle{definition}
\newtheorem{definition}[theorem]{Definition}
\newtheorem{example}[theorem]{Example}

\theoremstyle{remark}
\newtheorem{remark}[theorem]{Remark}
\newtheorem{problem}[theorem]{Problem}

\title{\textbf{On combinatorial bounds for the total Tjurina numbers of certain curves and surfaces with isolated singularities}}
\author{Piotr Pokora}
\date{\today}

\begin{document}
\maketitle

\begin{abstract}
We investigate combinatorial bounds for the total Tjurina numbers of some plane curve arrangements. Focusing on arrangements of lines and conics in $\mathbb{P}^2$ that admit only ordinary quasi-homogeneous singularities, we derive new structural inequalities governing the distribution of multiple intersection points. As a consequence, we establish sharp lower bounds for the total Tjurina numbers of free line arrangements with bounded maximal multiplicity and, more generally, for free conic-line arrangements. In particular, we show that for a free arrangement of $d$ lines and $k$ conics, the total Tjurina number grows at least quadratically in $d$ and $k$, and we demonstrate that this bound is sharp.

As an application of these planar results, we construct a special family of surfaces in $\mathbb{P}^{3}$ with only isolated singularities and arbitrarily large total Tjurina numbers. This provides new lower bounds for the total Tjurina numbers of certain hypersurfaces that are independent of detailed homological data.
\end{abstract}
\section{Introduction}

The total Tjurina number of a projective hypersurface measures the complexity of its singularities and plays a central role in the interplay between singularity theory, combinatorics, and the homological properties of Jacobian algebras. For plane curves, a classical result of du~Plessis and Wall establishes sharp upper bounds on the total Tjurina number in terms of the degree and the minimal degree of Jacobian relations, with equality characterizing the so-called maximal Tjurina curves \cite{DPP1}. While these upper bounds are rather well understood, see for instance \cite{Dimca1}, much less is known about effective lower bounds, especially in geometrically constrained situations such as line or conic-line arrangements with prescribed singularities.
The aim of this paper is to establish combinatorial lower bounds for the total Tjurina numbers of plane curve arrangements with only ordinary quasi-homogeneous singularities, with a particular focus on free arrangements. This work continues the line of research we initiated recently in \cite{Pok}, where we studied the freeness of conic-line arrangements with ordinary quasi-homogeneous singularities, and shifts the focus toward effective lower bounds for their total Tjurina numbers. Our results are motivated by recent developments in the theory of plane curves: by imposing natural geometric conditions, we can observe that free arrangements are highly constrained. We first derive structural inequalities for complex line arrangements with bounded multiplicities, demonstrating that the combinatorics of intersection points strongly restrict possible singular configurations. In particular, Theorem \ref{super} implies that there are only finitely many combinatorial types of supersolvable line arrangements admitting points of multiplicity $\leq 4$. We then consider free conic-line arrangements with ordinary quasi-homogeneous singularities, proving that their total Tjurina numbers grow at least quadratically in the number of components. This constitutes the main result of the paper, see Theorem \ref{main}.

A further aim is to extend these planar results to higher dimensions. We construct special surfaces in $\mathbb{P}^3$ with only isolated singularities and arbitrarily large total Tjurina numbers, obtaining bounds independent of delicate homological data, see Proposition \ref{sur}. In this way, the paper connects combinatorial geometry of arrangements with the study of singular hypersurfaces, providing new tools to bound and construct varieties with controlled singular behavior.
\section{Preliminaries}

Throughout the paper we work over the field of complex numbers $\mathbb{C}$. All curves are assumed to be reduced and projective. We follow the notation introduced in \cite{Dimca}.
Let $S = \mathbb{C}[x,y,z]$ be the graded polynomial ring in three variables $x,y,z$.

We start with general combinatorial preliminaries. Let $C \subset \mathbb{P}^2$ be an arrangement of $k$ smooth curves that admits only ordinary singularities. For each $r \geq 2$, we denote by $n_r(C)=n_{r}$ the number of ordinary singular points of multiplicity $r$, i.e., points where exactly $r$ curves from $C$ meet transversally. The \textbf{maximal multiplicity} of the curve $C$ is defined as
\[
m(C) := \max \{ r : n_r(C) \neq 0 \}.
\]
Now we recall the basic notions concerning the Tjurina number of a plane curve singularity.

Let $C \subset \mathbb{P}^2$ be a curve defined by a homogeneous polynomial 
$f \in \mathbb{C}[x,y,z]$, and let $p \in {\rm Sing}(C)$ be a singular point. Since the problem is local in nature, we may assume (after choosing suitable affine coordinates) that 
$p = (0,0) \in \mathbb{C}^2$. Thus, in a sufficiently small neighborhood of $p$, the curve $C$ is given by a convergent power series
\[
f(x,y) \in \mathbb{C}\{x,y\},
\]
and we may treat $C$ locally as a plane curve in two variables. 

\begin{definition}
The \textbf{Tjurina number} of $C$ at $p = (0,0)$ is defined by
\[
\tau_{p}(C) = \dim_{\mathbb{C}} \frac{\mathbb{C}\{x,y\}}{\langle f, f_x, f_y\rangle}.
\]
The \textbf{total Tjurina number} is defined as
\[
\tau(C) = \sum_{p \in \operatorname{Sing}(C)} \tau_{p}(C).
\]
\end{definition}
\noindent
If $p \in {\rm Sing}(C)$ is an \textbf{ordinary quasi-homogeneous singularity} of multiplicity $r$, then
\[
\tau_{p}(C) = (r-1)^2,
\]
where $r$ denotes the multiplicity of $C$ at $p$, equivalently, the number of local branches passing through $p$. This follows from the equality $\tau_p = \mu_p$, see \cite{KS}. Hence, if $C \subset \mathbb{P}^{2}$ is a reduced plane curve admitting only ordinary quasi-homogeneous singularities, then its total Tjurina number is given by
\[
\tau(C) = \sum_{p \in \operatorname{Sing}(C)} \tau_{p}(C)
= \sum_{r \ge 2} (r-1)^2 n_r.
\]

Finally, let us define the notion of free plane curves.  Let $C = \{f = 0\} \subset \mathbb{P}^2$ be a reduced plane curve of degree $m$. Denote by
\[
AR(f) := \{(a,b,c) \in S^3 : af_x + bf_y + cf_z = 0\}
\]
the module of Jacobian syzygies.
\begin{definition}
We say that a reduced plane curve $C$ of degree $m$ is \textbf{free} if $AR(f)$ is a free graded $S$-module of rank $2$. In this case there exist integers $(d_1, d_2)$,
called the \textbf{exponents} of $C$, satisfying $d_1 + d_2 = m - 1$.
\end{definition}
For free curves we have the following fundamental characterization. Let us define the minimal degree of Jacobian relations of $C = \{f=0\}$ as
$${\rm mdr}(f) :={\rm min }\{r : {\rm AR}(f)_{r}\neq 0\}.$$
\begin{theorem}[{du Plessis -- Wall, \cite{DPP1}}]
\label{dup}
Let $C = \{f=0\}\subset \mathbb{P}^{2}$ be a reduced curve of degree $m$ and $d_1 = {\rm mdr}(f)$. Let us denote the total Tjurina number of $C$ by $\tau(C)$.
Then the following two cases hold.
\begin{enumerate}
\item[a)] If $d_1 < m/2$, then $\tau(C) \leq \tau_{max}(m,d_1 )= (m-1)(m-d_1-1)+d_{1}^2$ and the equality holds if and only if the curve $C$ is free.
\item[b)] If $m/2 \leq d_{1} \leq m-1$, then
$\tau(C) \leq \tau_{max}(m,d_{1})$,
where, in this case, we set
$$\tau_{max}(m,d_{1})=(m-1)(m-d_{1}-1)+d_{1}^2- \binom{2d_{1}-m+2}{2}.$$
\end{enumerate}
\end{theorem}

\section{Results concerning complex line arrangements with small multiplicities}
Our first result is a structural statement concerning the intersection properties of complex line arrangements. It was originally proved in the dual setting of $r$-rich lines \cite[Proposition 31]{PP1}. However, we present it here in a form adapted to the framework of this paper and provide a brief proof.
\begin{proposition}
Let $\mathcal{L} \subset \mathbb{P}^{2}$ be an arrangement of $d\geq 6$ lines such that $n_{r} = 0$ for all $r > \frac{2}{3}d$. Then one has
\[n_{2} + n_{3} + n_{4} \geq \frac{d(d+15)}{18}.\]
\end{proposition}
\begin{proof}
We begin with the observation that, in the present setting, the following Hirzebruch-type inequality holds (see \cite[Section 11]{Langer} or \cite[Remark 2.4]{PP}):
\begin{equation}
\label{hir}
    n_{2} + \frac{3}{4}n_{3} \geq d + \sum_{r\geq 5}\frac{r(r-4)}{4}n_{r}.
\end{equation}
Next, for every $r\geq 5$ we have 
\[\frac{r(r-4)}{4} \geq \frac{1}{8} \cdot\frac{r(r-1)}{2}.\]
Let us look at the na\"ive combinatorial count in the following form:
$$\binom{d}{2} - n_{2} - 3n_{3} - 6n_{4} = \sum_{r\geq5} \binom{r}{2}n_{r}.$$
Combining this identity with the Hirzebruch-type inequality \eqref{hir}, we obtain
$$n_{2}+\frac{3}{4}n_{3} \geq d + \sum_{r\geq 5}\frac{r(r-4)}{4}n_{r} \geq d + \frac{1}{8}\bigg(\binom{d}{2}-n_{2}-3n_{3} - 6n_{4}\bigg),$$
hence
$$\frac{9}{8}(n_{2} + n_{3}+n_{4}) \geq \frac{9}{8}n_{2} + \frac{9}{8}n_{3} + \frac{6}{8}n_{4} \geq \frac{d(d+15)}{16},$$
which completes the proof.
\end{proof}
Now we present an interesting application of the above result in the context of free line arrangements with small multiplicities. The first result is general and it works without the freeness assumption.
\begin{corollary}
\label{bbb}
Let $\mathcal{L} \subset \mathbb{P}^{2}$ be an arrangement of $d\geq 6$ lines such that $m(\mathcal{L})=4$. Then
$$n_{3}+3n_{4} \leq \frac{4d(d-3)}{15}.$$
\end{corollary} 
\begin{proof}
We start with the na\"ive combinatorial count:
\[ \frac{d(d-1)}{2} = n_{2} + n_{3} + n_{4} + 2n_{3}+5n_{4} \geq n_{2}+n_{3}+n_{4} + \frac{5}{3}(n_{3}+3n_{4})\geq \frac{d(d+15)}{18}+\frac{5}{3}(n_{3}+3n_{4}),\]
hence we get
$$\frac{5}{3}(n_{3}+3n_{4}) \leq \frac{8d^{2}-24d}{18},$$
and this completes the proof.
\end{proof}
\begin{proposition}
Let $\mathcal{L} \subset \mathbb{P}^{2}$ be a free arrangement of $d\geq 4$ lines such that $m(\mathcal{L})=4$. Then
\[\frac{(d-1)(d-3)}{4} \leq n_{3}+3n_{4}.\]
\end{proposition}
\begin{proof}
Recall that \cite[Proposition 3.1]{JL} shows that, for a free line arrangement with $m(\mathcal L)=4$, one has
$$n_{2} + n_{3} \leq \frac{3(d-1)}{2}.$$
Therefore,
$$\frac{d(d-1)}{2}  = n_{2} + n_{3} + 2(n_{3}+3n_{4}) \leq \frac{3(d-1)}{2} + 2(n_{3}+3n_{4}),$$
and the claimed inequality follows.
\end{proof}
\begin{remark}
The above considerations show that if 
$\mathcal{L} \subset \mathbb{P}^{2}$ is a free arrangement of $d \ge 6$ lines with $m(\mathcal{L}) = 4$, then
\[\frac{(d-1)(d-3)}{4} \le n_{3} + 3n_{4} \le \frac{4d(d-3)}{15}.\]
Both estimates are close to being attained.

Indeed, consider the simplicial line arrangement $\mathcal{A}_{1}(9)$, which consists of $d=9$ lines and has
$n_{2}=6$, $n_{3}=4$, and $n_{4}=3$. We know that $\mathcal{A}_{1}(9)$ is free. In this case,
\[12 = \frac{(9-1)(9-3)}{4} \le n_{3} + 3n_{4} = 13.\]

The upper bound is also quite accurate. For the Hesse arrangement of $12$ lines with
$n_{2} = 12$ and $n_{4} = 9$, we obtain
\[27 \le \left\lfloor \frac{4d(d-3)}{15} \right\rfloor = 28.\]

These examples motivate the question whether the degrees of free complex line arrangements with maximal multiplicity~$4$ are bounded.
\end{remark}
To conclude this section, we present a result that further narrows the possible combinatorics of free arrangements of lines $\mathcal{L}$ with $m(\mathcal{L})=4$. Recall that an arrangement $\mathcal{L} \subset \mathbb{P}^{2}$ is called \textbf{supersolvable} if it admits a modular intersection point. It is well-known that every supersolvable arrangement is free \cite{TJ}.
\begin{theorem}
\label{super}
If $\mathcal{L}$ is a supersolvable arrangement of $d\geq 6$ lines with $m(\mathcal{L})=4$, then $d \in \{6,7,8,9,10\}$.
\end{theorem}
\begin{proof}
By \cite[Lemma~2.1]{AT}, every point of maximal multiplicity in a supersolvable line arrangement is modular. Since $m(\mathcal L)=4$, the arrangement $\mathcal L$ therefore admits a modular point of multiplicity~$4$, and by the result quoted in \cite[Theorem~1.12(1)]{Sol}, its exponents are $(3,d-4)$, up to order. Hence
\[
\tau(\mathcal L)=(d-1)^2-3(d-4).
\]
On the other hand, a straightforward combinatorial count gives
\[\tau(\mathcal{L}) = n_{2} + 4n_{3} + 9n_{4}
= \frac{d(d-1)}{2} + n_{3} + 3n_{4}.\]
Comparing these two expressions for $\tau(\mathcal{L})$, we obtain
\[n_{3} + 3n_{4} = \frac{d^{2} - 9d + 26}{2}.\]
Finally, Corollary~\ref{bbb} yields the inequality
\[n_{3} + 3n_{4} \leq \frac{4d(d-3)}{15}.\]
Combining the above relations shows that $d \in \{6,7,8,9,10\}$, which completes the proof.
\end{proof}
\begin{corollary}
There are only finitely many combinatorial types of supersolvable line arrangements $\mathcal{L}$ with $m(\mathcal{L})=4$.
\end{corollary}

\section{A lower bound on the total Tjurina numbers of some conic-line arrangements}
We now prove the main result of the paper.
\begin{theorem}
\label{main}
Let $\mathcal{CL} \subset \mathbb{P}^{2}$ be a free arrangement of $d$ lines and $k$ conics admitting only ordinary quasi-homogeneous singularities. Then
$$\tau(\mathcal{CL}) \geq 3k(k-1) + 3kd + \frac{3(d-1)^2}{4}.$$
Moreover, this bound is sharp.
\end{theorem}
\begin{proof}
Let ${\rm deg}(\mathcal{CL}) = 2k+d$ and $d_{1} = {\rm mdr}(\mathcal{CL})$. The freeness of $\mathcal{CL} \subset \mathbb{P}^{2}$ implies that the following equality holds:
\begin{equation}
\label{dpp}
d_{1}^{2} - d_{1}(2k+d-1) +(2k+d-1)^2 = \tau(\mathcal{CL}) = \sum_{r\geq 2}(r-1)^{2}n_{r}.
\end{equation}
Recall that the following na\"ive combinatorial count holds:
\begin{equation}
\label{ncc}
 4\binom{k}{2}+2kd +\binom{d}{2} = \sum_{r\geq 2}\binom{r}{2}n_{r},
\end{equation}
which can be written as
\begin{equation}
(2k+d-1)^{2}+d-1 = f_{2}-f_{1}
\end{equation}
where $f_{i}=\sum_{r\geq 2}r^{i}n_{r}$.
Let us come back to \eqref{dpp}. We can rewrite it as
\[d_{1}^{2}-d_{1}(2k+d-1)+(2k+d-1)^{2} = f_{2}-2f_{1}+f_{0} = (2k+d-1)^{2}+d-1-f_{1}+f_{0}\]
and hence we arrive at
\[d_{1}^{2}-d_{1}(2k+d-1) + \bigg(\sum_{r\geq 2}(r-1)n_{r}-d+1 \bigg) = 0.\]
We compute the discriminant of the above equation in $d_{1}$: 
\begin{equation}
\label{cruc}
\triangle_{d_{1}} := (2k+d-1)^2 - 4(1-d) -4\sum_{r\geq 2}(r-1)n_{r}.
\end{equation}
Since $\mathcal{CL}$ is free we have $\triangle_{d_{1}}\geq 0$, and this gives us
\[(2k+d)^2 + 2d-4k-3 \geq \sum_{r\geq 2}(4r-4)n_{r}= (2k+d)^2-4k-d-\sum_{r\geq 2}(r^{2}-5r+4)n_{r}.\]
Hence
\begin{equation}
\label{n5}
\frac{3}{2}(d-1) + \frac{1}{2}\sum_{r\geq 5}(r^{2}-5r+4)n_{r} \geq n_{2}+n_{3}.
\end{equation}
Let us focus again on \eqref{ncc}, we have
\[\frac{4k(k-1)+4kd+d(d-1)}{2} = \sum_{r\geq 2}\frac{r(r-1)}{2}n_{r} = n_{2} + n_{3} + 2n_{3} + \sum_{r\geq 4} \frac{r(r-1)}{2}n_{r} \stackrel{\eqref{n5}}{\leq} \]
\[\frac{3}{2}(d-1) + 2n_{3} + \sum_{r\geq 4}\frac{r^{2}-5r+4}{2}n_{r} + \sum_{r\geq 4}\frac{r(r-1)}{2}n_{r} = \]
\[\frac{3}{2}(d-1)+2n_{3} + \sum_{r\geq 4}(r^{2}-3r+2)n_{r},\]
hence
\begin{equation}
\label{nn}
\frac{4k(k-1)+4kd+d^{2}-4d+3}{4}\leq n_{3} + \frac{1}{2}\sum_{r\geq 4}(r^{2}-3r+2)n_{r}.
\end{equation}
We now turn to estimating the total Tjurina number of $\mathcal{CL}$. We have
\begin{multline*}
\tau(\mathcal{CL}) = \sum_{r\geq 2}(r-1)^{2}n_{r} = \sum_{r\geq 2}\frac{r(r-1)}{2}n_{r} + \sum_{r \geq 2}\frac{r^{2}-3r+2}{2}n_{r} \stackrel{\eqref{nn}}{\geq} \\ \frac{4k(k-1)+4kd+d(d-1)}{2} + \frac{4k(k-1)+4dk+d^2-4d+3}{4} = \\ 3k(k-1)+3kd + \frac{3(d-1)^{2}}{4}. 
\end{multline*}
To see that the bound is sharp, let us recall that there exists a unique (up to projective equivalence) free arrangement $\mathcal{CL}_{5}$ of $d=3$ lines and $k=1$ conic such that $n_{3}=3$, see \cite[Example~4.14 and Theorem~5.7]{DP22}. In this case we have $\tau(\mathcal{CL}_{5})=12$, and
\[
3k(k-1) + 3kd + \frac{3(d-1)^2}{4}
= 0 + 9 + \frac{3\cdot4}{4}
= 12.
\]
Hence we obtain equality, which shows that the bound is sharp.
\end{proof}
\begin{corollary}
If $\mathcal{L} \subset \mathbb{P}^{2}$ is a free arrangement of $d\geq 2$ lines, then
$$\tau(\mathcal{L}) \geq \frac{3(d-1)^2}{4}.$$
\end{corollary}
The obtained lower bound is quite accurate, as the following example shows. Let us recall that the Klein arrangement $\mathcal{K}$ of $21$ lines is a free line arrangement such that $n_{3}=28$ and $n_{4}=21$, see \cite[Chapter 6]{OT92}.
\begin{example}
For $d=21$ we have $\tau(\mathcal{L})\geq 300$. We can easily compute that for the Klein arrangement $\mathcal{K}$ of $21$ lines we have $\tau(\mathcal{K}) = 301$.
\end{example}
Based on the above considerations, we can formulate the following difficult problem.
\begin{problem}
\label{prob}
Classify geometrically free line arrangements $\mathcal{L} \subset \mathbb{P}^2$
with fixed degree $d = |\mathcal{L}|$ and fixed maximal multiplicity
$m(\mathcal{L})$. In particular, determine all possible weak combinatorial types and the corresponding Tjurina numbers $\tau(\mathcal{L})$ under these constraints.
\end{problem}
As a concrete example, consider the case $d = 21$ with $m(\mathcal{L}) = 4$. It is known that
\[\tau(\mathcal{L}) \in \{300,301\},\]
and that there exist $21$ weak combinatorial types of such arrangements. Among these is the Klein arrangement. A natural question therefore arises: which of the remaining combinatorial types, aside from the Klein combinatorics, are geometrically realizable as free arrangements?

\begin{corollary}
Let $\mathcal{L}\subset \mathbb{P}^{2}$ be a free arrangement of $d$ lines such that $n_{d} = 0$ with exponents $(d_{1},d_{2})$. Then $$d_{1}d_{2} \leq \frac{(d-1)^{2}}{4}.$$
Moreover, this bound is sharp if and only if $d_{1}=d_{2}$.
\end{corollary}
\begin{proof}
Since $\mathcal{L}$ is free we get $(d-1)^{2} = d_{1}d_{2}+\tau(\mathcal{L}) \geq d_{1}d_{2} + \frac{3}{4}(d-1)^{2}$.
Moreover, if $d_{1}=d_{2}$, then $2d_{1}=d-1$, and the equality follows.
\end{proof}
\begin{example}
The dual Hesse arrangement $\mathcal{H}$ consists of $d=9$ lines with $n_{3}=12$ triple points. 
We know that $\mathcal{H}$ is free with exponents $(d_{1},d_{2})=(4,4)$. Moreover,
\[
d_{1}d_{2}=4\cdot4=16
\quad\text{and}\quad
\frac{(d-1)^{2}}{4}=\frac{8^{2}}{4}=16.
\]
Hence $d_{1}d_{2}=\frac{(d-1)^{2}}{4}$, so the bound is sharp.
\end{example}
\section{An application to surfaces with only isolated singularities}
We apply our lower bound on the total Tjurina number of conic-line arrangements with quasi-homogeneous singularities from Theorem \ref{main} to construct surfaces in $\mathbb{P}^3$ with only isolated singularities and arbitrarily large total Tjurina numbers. 

Let us denote by $R = \mathbb{C}[x,y,z,w]$ the graded ring of polynomials in four variables. Recall that for a reduced surface $X = \{F(x,y,z,w)=0\} \subseteq \mathbb{P}^{3}$ we define
$$d_{1}' = {\rm mdr}(F) = {\rm min}\{r : {\rm AR}(F)_{r}\neq 0\},$$
where
$${\rm AR}(F) = \{ (a_{1},a_{2},a_{3},a_{4})\in R^{4} : a_{1}\partial_{x}F + a_{2}\partial_{y}F + a_{3}\partial_{z}F + a_{4}\partial_{w}F = 0\}.$$
We have the following general result due to du Plessis and Wall \cite[Theorem 5.4]{DPP2}.
\begin{theorem}
If $X = \{F(x,y,z,w)=0\}\subset\mathbb{P}^{3}$ is a reduced surface of degree $D$ with only isolated singularities, then
\[(D-1)^{3}-d_{1}'(D-1)^{2} \leq \tau(X) \leq (D-1)^{3}-d_{1}'(D-d_{1}'-1)(D-1).\]
\end{theorem}
In light of the above result, we present our contribution toward bounding the total Tjurina number of certain surfaces that is in the spirit of \cite[Remark 3.2]{DS26}. The main advantage of our result is that the bound does not depend explicitly on the homological data associated with the surface.
\begin{proposition}
\label{sur}
Let $C = \{f(x,y,z) =0\} \subset \mathbb{P}^{2}$ be an arrangement of $k$ conics and $d$ lines with only ordinary quasi-homogeneous singularities. Assume that $C$ is free. Consider the surface
$$X = \{F(x,y,z,w) = f(x,y,z)+w^{2k+d} = 0\} \subset \mathbb{P}^{3}.$$
Then $X$ has only isolated singularities and
$$\tau(X) \geq (2k+d-1)\cdot\bigg(3k(k-1) + 3kd + \frac{3(d-1)^2}{4}\bigg).$$
\end{proposition}
\begin{proof}
The fact that $X$ has only isolated singularities follows from straightforward computations of its partial derivatives. To obtain a lower bound on 
$\tau(X)$, we note that the problem is local. Hence, we perform computations around each singular point, and we have a na\"ive Thom-Sebastiani principle for isolated singularities, namely
$$\tau\bigg(f(x,y,z) + w^{2k+d}\bigg) = \tau(f(x,y,z))\cdot\tau(w^{2k+d})= (2k+d-1)\cdot\tau(f(x,y,z)),$$
which completes the proof.
\end{proof}
%%%%%%%%%%%%%%%%%%%%%%%%%%%%%%%%%%%%%%%%%%%%%%%%%%%%%%%%%%%%%%%%%%%%%%%%%%%%%%%%%%%%%%%%%%%%%%%%%%%%%%%%%%%%%%%%%%%%%%%%%%%%%%%%%%%%%%%%%%%%%%%%%%%%%%%%%%%%%%%%%%%%%
\section*{Acknowledgments}
I would like to warmly thank Alex Dimca for all comments regarding the content of the paper and for suggesting Theorem \ref{super}.
\section*{Funding}
Piotr Pokora is supported by Narodowe Centrum Nauki (Poland) Sonata Bis Grant  
\begin{center}
\textbf{2023/50/E/ST1/00025}.    
\end{center} For the purpose of Open Access, the author has applied a CC-BY public copyright license to any Author Accepted Manuscript (AAM) version arising from this submission.

\section*{Conflict of interest statement}
There are no conflicts of interest to declare.

\section*{Data availability statement}
Data sharing is not applicable to this article as no data sets were generated or analyzed during this study.
%%%%%%%%%%%%%%%%%%%%%%%%%%%%%%%%%%%%%%%%%%%%%%%%%%%%%%%%%%%%%%%%%%%%%%%%%%%%%%%%%%%%%%%%%%%%%%%%%%%%%%%%%%%%%%%%%%%%%%%%%%%%%%%%%%%%%%%%%%%%%%%%%%%%%%%%%%%%%%%%%%%%%

\bigskip
%Affiliation of the author
Piotr Pokora\\
\noindent
Department of Mathematics,
University of the National Education Commission Krakow,
Podchor\c a\.zych 2,
PL-30-084 Krak\'ow, Poland. \\
Email: \texttt{piotr.pokora@uken.krakow.pl}
\end{document}